\documentclass[a4paper,11pt]{amsart}
\usepackage{latexsym}
\usepackage{amsmath,amssymb}
\usepackage{amsmath}

\usepackage[all]{xy}
\usepackage{enumerate}
\usepackage[usenames]{color}
\usepackage{xcolor}

\newtheorem{theo}{Theorem}[section]
\newtheorem{coll}[theo]{Corollary}
\newtheorem{lemm}[theo]{Lemma}
\newtheorem{prop}[theo]{Proposition}
\newtheorem{defn}[theo]{Definition}
\newtheorem{ex}[theo]{Example}
\newtheorem{rem}[theo]{Remark}

\newcommand{\Hom}{{\rm Hom}}
\newcommand{\End}{{\rm End}}

\newcommand{\A}{\mathcal A}

\newcommand{\C}{\mathcal C}

\begin{document}
\sloppy

\title[Strongly CS-Rickart objects]{Strongly CS-Rickart and dual strongly CS-Rickart objects in abelian categories}

\author[S. Crivei]{Septimiu Crivei}

\address{Faculty of Mathematics and Computer Science, Babe\c s-Bolyai University, Str. M. Kog\u alniceanu 1, 400084 Cluj-Napoca, Romania} \email{crivei@math.ubbcluj.ro}

\author[S.M. Radu]{Simona Maria Radu}

\address{Faculty of Mathematics and Computer Science, Babe\c s-Bolyai University, Str. M. Kog\u alniceanu 1, 400084 Cluj-Napoca, Romania} \email{simonamariar@math.ubbcluj.ro}



\begin{abstract} We introduce (dual) strongly relative CS-Rickart objects in abelian categories, as common generalizations of (dual) strongly relative Rickart objects and strongly extending (lifting) objects. We give general properties, and we study direct summands, (co)products of (dual) strongly relative CS-Rickart objects and classes all of whose objects are (dual) strongly self-CS-Rickart. 
\end{abstract}

\date{January 11, 2021}

\maketitle

\section{Introduction}

(Dual) CS-Rickart objects in abelian categories have been introduced by the authors in \cite{CR1} as a common generalization of (dual) Rickart objects and extending (lifting) objects. In modules categories, they have been introduced and studied by Abyzov, Nhan and Quynh \cite{AN1,ANQ} and Tribak \cite{Tribak}. Rickart objects and their duals in abelian categories have been introduced and studied by Crivei, K\"or and Olteanu \cite{CK,CO}, and subsume previous work of D\u asc\u alescu, N\u ast\u asescu, Tudorache and D\u au\c s \cite{DNTD} on regular objects in abelian categories, Lee, Rizvi and Roman \cite{LRR10,LRR11} on Rickart and dual Rickart modules, and in particular, Rizvi and Roman \cite{RR04,RR09} and Keskin T\"ut\"unc\"u and Tribak \cite{KT} on Baer and dual Baer modules. The study of Baer modules, Rickart modules and their duals have the origin in the work of von Neumann on regular rings \cite{vN}, Kaplansky \cite{K} on Baer rings and Maeda \cite{Maeda} on Rickart rings. On the other hand, extending and lifting modules have been important concepts in module theory for the last decades (e.g., see the monographs \cite{CLVW,DHSW}). The reader is referred to our paper \cite{CR1} for further background on CS-Rickart related notions.

The concepts of (dual) Rickart objects and extending (lifting) objects in abelian categories may be specialized to those of (dual) strongly Rickart objects and strongly extending (lifting) objects. In this direction, we mention the work of Crivei and Olteanu on (dual) strongly Rickart objects \cite{CO1,CO2}, Al-Saadi and Ibrahiem on (dual) strongly Rickart modules \cite{AI14,AI15}, Ebrahimi Atani, Khoramdel and Dolati Pish Hesari \cite{EKD} on strongly extending modules or Wang \cite{Wang} on strongly lifting modules. All these notions are obtained from the corresponding general notions by referring to fully invariant direct summands instead of direct summands. 

As (dual) CS-Rickart objects  generalize (dual) Rickart and exending (lifting) objects, it is natural to consider the concept of (dual) strongly CS-Rickart object, as a common generalization of (dual) strongly Rickart objects and strongly extending (lifting) objects in abelian categories. This is the topic of the present paper. As in our previous related work, we take advantage of the setting of abelian categories in order to extensively apply the duality principle and automatically obtain dual properties. Of course, the proofs of the results will only be given for one of the dual notions.    

We briefly present our main results. In Section 2 we introduce and give the first properties of (dual) strongly relative CS-Rickart objects in abelian categories. We show how 
our concepts relate to those of (dual) relative CS-Rickart objects and (dual) strongly relative Rickart objects. For instance, an object $M$ of an abelian category $\mathcal{A}$ is (dual) strongly self-CS-Rickart if and only if $M$ is (dual) self-CS-Rickart and weak duo if and only if $M$ is (dual) self-CS-Rickart and ${\rm End}_{\mathcal{A}}(M)$ is abelian. We also present examples supporting our theory. 
In Section 3 we show that the class of (dual) relative CS-Rickart objects is closed under direct summands. We also prove that every strongly self-CS-Rickart object satisfies a stricter form of the SIP-extending (SIP-lifting) property on direct summands. 
In Section 4 we study (co)products of (dual) strongly relative CS-Rickart objects. In this direction, we show that if $M$ and $N_1,\dots,N_n$ are objects of an abelian category $\mathcal{A}$, then $\bigoplus_{i=1}^n N_i$ is (dual) strongly $M$-CS-Rickart if and only if $N_i$ is (dual) strongly $M$-CS-Rickart for every $i\in \{1,\dots,n\}$. Also, for a (finite) direct sum decomposition $M=\bigoplus_{i\in I}M_i$ in an abelian category $\mathcal{A}$, we prove that $M$ is (dual) strongly self-CS-Rickart if and only if $M_i$ is (dual) strongly self-CS-Rickart for each $i\in I$ and $\Hom_{\mathcal{A}}(M_i,M_j)=0$ for every $i,j\in I$ with $i\neq j$. We end with the structure of some (dual) strongly self-CS-Rickart modules over a Dedekind domain. Finally, in Section 5 we study classes all of whose objects are (dual) strongly self-CS-Rickart. Among other results, for an abelian category $\A$ with enough injectives (projectives), and a class $\mathcal{C}$ of objects of $\A$ which is closed under binary direct sums and contains all injective (projective) objects of $\A$, we prove that every object of $\C$ is strongly extending (lifting) if and only if every object of $\C$ is (dual) strongly self-CS-Rickart.

\section{(Dual) relative strongly CS-Rickart objects}

For every morphism $f:M\to N$ in an abelian category $\mathcal{A}$ we denote by
${\rm ker}(f):{\rm Ker}(f)\to M$, ${\rm coker}(f):N\to {\rm Coker}(f)$ and ${\rm im}(f):{\rm Im}(f)\to N$ the kernel, the cokernel and the image of $f$ respectively. 
For a short exact sequence $0\to A\to B\to C\to 0$ in $\A$, we sometimes write $C=B/A$. A morphism $f:A\to B$ is called a \emph{section} (\emph{retraction}) if there is a morphism
$f':B\to A$ such that $f'f=1_A$ ($ff'=1_B$). We also recall the following notions, which are natural generalization from modules \cite{CO1,Wis}.

\begin{defn}\rm Let $\mathcal{A}$ be an abelian category. 
\begin{enumerate}
\item \rm A monomorphism $f:K \to M$ in $\mathcal{A}$ is called: 
\begin{enumerate}[(i)]
\item \textit{essential} if for every morphism $h:M \to P$ in $\mathcal{A}$ such that $hf$ is a monomorphism, $h$ is a monomorphism. 
\item \emph{fully invariant} if for every morphism $h:M\to M$,
there exists a morphism $\alpha:K\to K$ such that $hf=f\alpha$.
\end{enumerate}
\item \rm An epimorphism $g:M \to N$ in $\mathcal{A}$ is called: 
\begin{enumerate}[(i)]
\item \textit{superfluous} if for every morphism $h:Q \to M$ in $\mathcal{A}$ such that $gh$ is an epimorphism, $h$ is an epimorphism. 
\item \emph{fully coinvariant} if for every morphism $h:M\to M$,
there exists a morphism $\beta:N\to N$ such that $gh=\beta g$.
\end{enumerate}   
\end{enumerate}
\end{defn}

We recall the concepts of (dual) relative CS-Rickart objects, which generalize both (dual) relative Rickart objects and extending (lifting) objects in abelian categories.

\begin{defn}[{\cite[Definitions~2.3, 2.4]{CR1}}] \rm Let $M$ and $N$ be objects of an abelian category $\mathcal{A}$. Then $N$ is called:
\begin{enumerate}
\item 
\begin{enumerate}[(i)]
\item \rm {\textit{$M$-CS-Rickart}} if for every morphism $f:M\to N$ there are an 
essential monomorphism $e:{\rm Ker}(f) \to L$ and a section $s:L \to M$ in $\mathcal{A}$ such that ${\rm ker}(f)=se$. Equivalently, $N$ is $M$-CS-Rickart if and only if for every morphism $f:M\to N$, ${\rm Ker}(f)$ is essential in a direct summand of $M$.
\item \rm {\textit{self-CS-Rickart}} if $N$ is $N$-CS-Rickart.
\item \emph{extending} if every subobject of $M$ is essential in a direct summand of $M$.
\end{enumerate}
\item 
\begin{enumerate}[(i)]
\item \rm{\textit{dual $M$-CS-Rickart}} if for every morphism $f:M\to N$ there are a    
retraction $r:N \to P$ and a superfluous epimorphism $t:P\to{\rm Coker}(f)$ in $\mathcal{A}$ such that ${\rm coker}(f)=tr$.
Equivalently, $N$ is dual $M$-CS-Rickart if and only if for every morphism $f:M\to N$, 
${\rm Im}(f)$ lies above a direct summand of $N$, in the sense that ${\rm Im}(f)$ contains a direct summand $K$ of $M$ such that ${\rm Im}(f)/K$ is superfluous in $M/K$. 
\item \rm  {\textit{dual self-CS-Rickart}} if $N$ is dual $N$-CS-Rickart.
\item \emph{lifting} if every subobject $L$ of $M$ lies above a direct summand of $M$, 
in the sense that $L$ contains a direct summand $K$ of $M$ such that $L/K$ is superfluous in $M/K$. 
\end{enumerate}
\end{enumerate}
\end{defn}

Now we have the following specializations of the above notions by restricting to fully invariant sections and fully coinvariant retractions. Such notions have already been considered in module categories \cite{AN1,ANQ,EKD,Wang}.

\begin{defn} \rm Let $M$ and $N$ be objects of an abelian category $\mathcal{A}$. Then $N$ is called:
\begin{enumerate}
\item 
\begin{enumerate}[(i)]
\item \rm {\textit{strongly $M$-CS-Rickart}} if for every morphism $f:M\to N$ there are an essential monomorphism $e:{\rm Ker}(f) \to L$ and a fully invariant section $s:L \to M$ in $\mathcal{A}$ such that ${\rm ker}(f)=se$. Equivalently, $N$ is strongly $M$-CS-Rickart if and only if for every morphism $f:M\to N$, ${\rm Ker}(f)$ is essential in a fully invariant direct summand of $M$.
\item \rm {\textit{self-CS-Rickart}} if $N$ is strongly $N$-CS-Rickart.
\item \emph{strongly extending} if every subobject of $M$ is essential in a fully invariant direct summand of $M$.
\end{enumerate}
\item 
\begin{enumerate}[(i)]
\item \rm{\textit{dual $M$-CS-Rickart}} if for every morphism $f:M\to N$ there are a fully coinvariant  
retraction $r:N \to P$ and a superfluous epimorphism $t:P\to{\rm Coker}(f)$ in $\mathcal{A}$ such that ${\rm coker}(f)=tr$. Equivalently, $N$ is dual $M$-CS-Rickart if and only if for every morphism $f:M\to N$, ${\rm Im}(f)$ lies above a fully invariant direct summand of $N$.
\item \rm  {\textit{dual self-CS-Rickart}} if $N$ is dual $N$-CS-Rickart.
\item \emph{strongly lifting} if every subobject $L$ of $M$ lies above a fully invariant direct summand of $M$, in the sense that $L$ contains a fully invariant direct summand $K$ of $M$ such that $L/K$ is superfluous in $M/K$. 
\end{enumerate}
\end{enumerate}
\end{defn}

\begin{rem} \rm (1) Every (dual) strongly self-Rickart object and every strongly extending (lifting) object of an abelian category $\A$ is (dual) strongly self-CS-Rickart. 

(2) Let $M$ and $N$ be objects of an abelian category. If $M$ is uniform (i.e., every non-zero subobject of $M$ is essential), then $N$ is strongly $M$-CS-Rickart. If $N$ is hollow (i.e., every proper subobject of $N$ is superfluous), then $N$ is dual strongly $M$-CS-Rickart.
\end{rem}

Recall that an object $M$ of an abelian category $\mathcal{A}$ is called \emph{weak duo} if every section $K\to M$ is fully invariant, or equivalently, every retraction $M\to C$ is fully coinvariant \cite[Definition~2.6]{CO1}.

\begin{prop}\label{st0}
Let $M$ and $N$ be two objects of an abelian category $\mathcal{A}$.
    \begin{enumerate}
       \item Assume that every direct summand of $M$ is 
               isomorphic to a subobject of $N$. Then $N$ is strongly   
               $M$-CS-Rickart if and only if $N$ is $M$-CS-Rickart and  
               $M$ is weak duo.
       \item Assume that every direct summand of $N$ is   
                isomorphic to a factor object of $M$. Then $N$ is   
                dual strongly $M$-CS-Rickart if and only if $N$ is  
                dual  $M$-CS-Rickart and $N$ is weak duo.
   \end{enumerate}
\end{prop}

\begin{proof} (1) Assume that $N$ is strongly $M$-CS-Rickart. Then it is clearly $M$-CS-Rickart. In order to prove that $M$ is weak duo, we consider a direct summand $X$ of $M$ and write $M=X\oplus X'$. Then we have the composite morphism
$$\SelectTips{cm}{} 
\xymatrix{
  M\ar[r] ^{\pi_{X'}}&X'\ar@{=}[r]^{\psi}&Y\ar[r]^{i}& N
   }
,$$ 
where $\pi_{X'}:M\to X'$ is the canonical projection, $\psi$ is an isomorphism, $Y$ is a subobject of $N$ and $i$ is a monomorphism. Then ${\rm Ker}(i\psi \pi_{X'})={\rm Ker}(\pi_{X'})=X$ is essential in a fully invariant direct summand $U$ of $M$, because $N$ is strongly $M$-CS-Rickart. But $X$ is also a direct summand of $U$, which implies that $X=U$. Thus, $X$ is fully invariant in $M$.

Conversely, assume that $N$ is $M$-CS-Rickart and $M$ is weak duo. Let $f:M\to N$ be a morphism in $\mathcal{A}$. Then there  exists a direct summand $U$ of $M$ such that ${\rm Ker}(f)$ is essential in $U$. By the fact that $M$ is weak duo, we deduce that $U$ is a fully invariant direct summand of $M$, and therefore $N$ is strongly $M$-CS-Rickart.
\end{proof}

\begin{coll}\label{st00}
Let $M$ be an object of an abelian category $\mathcal{A}$. Then:
         \begin{enumerate}
                   \item $M$ is strongly self-CS-Rickart if and only if $M$ is    
                            self-CS-Rickart and weak duo.
                   \item $M$ is dual strongly self-CS-Rickart if and only if  
                            $M$ is dual self-CS-Rickart and weak duo.
          \end{enumerate}
\end{coll}

\begin{coll}\label{st01}
Let $M$ be an indecomposable object of an abelian category $\mathcal{A}$.Then:
           \begin{enumerate}
                    \item $M$ is strongly self-CS-Rickart if and only if $M$  
                             is self-CS-Rickart.
                    \item $M$ is dual strongly self-CS-Rickart if and only if 
                             $M$ is dual self-CS-Rickart.
\end{enumerate}
\end{coll}

Recall that a ring $R$ is called \emph{abelian} if every idempotent element of $R$ is central. An element $a$ of a ring $R$ is called \emph{left (right) semicentral} if $ba=aba$ ($ab=aba$) for every $b\in R$. A ring $R$ is abelian if and only if every idempotent element of $R$ is left semicentral if and only if every idempotent element of $R$ is right semicentral (e.g., see \cite[p.~17]{Wei}). 

\begin{prop}\label{st1}
Let $M$ be an object of an abelian category $\mathcal{A}$. Then:
         \begin{enumerate}
                  \item $M$ is strongly self-CS-Rickart if and only if $M$ is 
                           self-CS-Rickart and ${\rm End}_{\mathcal{A}}(M)$ is  
                           abelian.
                   \item $M$ is dual strongly self-CS-Rickart if and only if   
                            $M$ is dual self-CS-Rickart and 
                            ${\rm End}_{\mathcal{A}}(M)$ is abelian.
         \end{enumerate}
\end{prop}

\begin{proof}
            Assume first that $M$ is strongly self-CS-Rickart. Then $M$ is self-CS-Rickart. In order to show that the ring ${\rm End}_{\mathcal{A}}(M)$ is abelian, we prove that every idempotent element $e\in {\rm End}_{\mathcal{A}}(M)$ is left semicentral. Since every idempotent splits, there exist an object $K$ and morphisms $k:K\to M$ and $p:M\to K$ such that $e=kp$ and $pk=1_K$.
Since $M$ is strongly self-CS-Rickart, ${\rm Ker}(1-e)$ is essential in a fully invariant direct summand $L$ of $M$. Thus ${\rm Ker}(1-e)$ is a direct summand of $L$, as a direct summand of $M$. We deduce that ${\rm Im}(k)={\rm Im}(e)={\rm Ker}(1-e)=L$, and thus $k$ is fully  invariant. By \cite[Lemma~2.13]{CO1}, $e$ is left semicentral, and so ${\rm End}_{\mathcal{A}}(M)$ is abelian.

Conversely, assume that $M$ is self-CS-Rickart and ${\rm End}_{\mathcal{A}}(M)$ is abelian. In order to show that $M$ is strongly self-CS-Rickart, by Proposition \ref{st0} it is enough to prove that $M$ is weak duo. To this end, let $X$ be a direct summand of $M$. Then there exist a section $\sigma_1 :X\to M$ and a retraction $\pi_1 :M\to X$ such that $\pi_1\sigma_1=1_X$ and $\sigma_1\pi_1:M\to M$ is an idempotent. Since ${\rm End}_{\mathcal{A}}(M)$ is abelian, we deduce that $\sigma_1\pi_1$ is left semicentral. Thus, $\sigma_1$ is fully invariant by \cite[Lemma~2.13]{CO1}.
\end{proof}

\begin{ex} \label{ex1} \rm By Proposition \ref{st1} and \cite[Example~2.6]{CR1}, we immediately have the following examples in the corresponding module categories:

(i) The $\mathbb{Z}$-module $\mathbb{Z}_4=\mathbb{Z}/4\mathbb{Z}$ is both strongly self-CS-Rickart and dual strongly self-CS-Rickart, but neither strongly self-Rickart, nor dual strongly self-Rickart (also use \cite[Proposition~2.14]{CO1}).

(ii) The $\mathbb{Z}$-module $\mathbb{Z}_4$ is both strongly $\mathbb{Z}$-CS-Rickart and dual strongly $\mathbb{Z}$-CS-Rickart, but it is neither strongly $\mathbb{Z}$-Rickart, nor dual strongly $\mathbb{Z}$-Rickart.

(iii) The $\mathbb{Z}$-module $\mathbb{Z}$ is strongly self-CS-Rickart, but not dual strongly self-CS-Rickart. 
 
(iv) The $\mathbb{Z}$-module $\mathbb{Q}$ is dual strongly self-CS-Rickart, but not strongly lifting \cite[Example~2.14]{Tribak}.   

(v) The $\mathbb{Z}$-module $\mathbb{Z}\oplus \mathbb{Z}_p$ (for some prime $p$) is not strongly self-CS-Rickart, although it is self-CS-Rickart. Indeed, it is not weak duo, because for the morphism $f:\mathbb{Z}\oplus \mathbb{Z}_p\to \mathbb{Z}\oplus \mathbb{Z}_p$ defined by $f(m,n)=(0,m)$, we have $f(\mathbb{Z}\oplus 0)=0\oplus \mathbb{Z}_p$, which is not a submodule of $\mathbb{Z}\oplus 0$. Hence $\mathbb{Z}\oplus 0$ is a direct summand of $\mathbb{Z}\oplus \mathbb{Z}_p$, which is not fully invariant.
                
(vi) For relatively prime integers $m, n$, the $\mathbb{Z}$-module  $\mathbb{Z}_n$ is strongly $\mathbb{Z}_m$-CS-Rickart.

(vii) Let $K$ be a field and consider the ring $R=\begin{pmatrix}K&K[X]\\0&K[X] \end{pmatrix}$. Then $R$ is a self-Rickart right $R$-module, and thus it is a self-CS-Rickart right $R$-module. But ${\rm End}_R(R)\cong R$ is not abelian, hence 
$R$ is not a strongly self-CS-Rickart right $R$-module. It is neither a dual (strongly) self-CS-Rickart right $R$-module.
\end{ex}

\begin{ex} \rm Let $A$ be a ring, and let $G$ be a subgroup of the group ${\rm Aut}(A)$ of ring automorphisms of $A$. For $a\in A$ and $g\in G$, denote by $a^g$ the image of $a$ under $g$. Denote by $A^G$ the fixed ring of $A$ under $G$, i.e., $A^G=\{a\in A\mid a^g=a \quad \mbox{for every $g\in G$}\}$. The skew group ring is given by $A\ast G=\bigoplus_{g\in G}Ag$ with addition given componentwise and multiplication defined by $(ag)(bh)=ab^{g^{-1}}gh\in Agh$ for every $a, b\in A$
and $g, h\in G$. If $G=\{g_1, g_2,\ldots, g_n\}$, $a\in A$ and
$\beta=a_1g_1+a_2g_2+\ldots +a_ng_n\in A\ast G$ with $a_i\in A$, define $$a\cdot \beta=a^{g_1}a_1^{g_1}+a^{g_2}a_2^{g_2}+\ldots +a^{g_n}a_n^{g_n}.$$ Then $A$ is a right $A\ast G$-module.

Following \cite[Example 3.13]{LRR10}, consider the ring $A=\begin{pmatrix} \mathbb{Z}_2 & \mathbb{Z}_2 \\
0 & \mathbb{Z}_2 \end{pmatrix}$. If $g\in {\rm Aut}(A)$ is the conjugation by $\begin{pmatrix} 1 & 1 \\ 0 & 1 \end{pmatrix}$, it follows that $G=\{1,g\}$ is a subgroup of ${\rm Aut}(A)$. Consider the skew group ring $R=A\ast G$ and the right $R$-module $M=A$. Then we have:
\begin{align*}
\End_R(M)&=A^G=\left\{\begin{pmatrix} a&b \\ 0&a  \end{pmatrix}\Big | a,b\in \mathbb{Z}_2\right\}\\ &=\left\{\varphi_1=\begin{pmatrix}
0 & 0 \\
0 & 0 \\
\end{pmatrix}, \varphi_2=\begin{pmatrix}
0 & 1 \\
0 & 0 \\
\end{pmatrix}, \varphi_3=\begin{pmatrix}
1 & 0 \\
0 & 1 \\
\end{pmatrix}, \varphi_4=\begin{pmatrix}
1 & 1 \\
0 & 1 \\
\end{pmatrix}\right\}.
\end{align*}
It follows that: $${\rm Ker} (\varphi_1)=M, \quad {\rm Ker} (\varphi_2)=\begin{pmatrix}
\mathbb{Z}_2 & \mathbb{Z}_2 \\
0 & 0 \\
\end{pmatrix}, \quad {\rm Ker} (\varphi_3)=0, \quad {\rm Ker} (\varphi_4)=\begin{pmatrix}
0 & \mathbb{Z}_2 \\
0 & 0 \\
\end{pmatrix}$$ \cite[Example~2.3]{CK18}. Now it is easy to see that the right $R$-module $M$ is self-CS-Rickart, and its endomorphism ring is  abelian. Hence the right $R$-module $M$ is strongly self-CS-Rickart by Proposition \ref{st01}. But it is not strongly self-Rickart \cite[Example 3.13]{LRR10}.
\end{ex}

Now recall the following concepts, which will be useful to relate (dual) strongly relative Rickart and (dual) strongly relative CS-Rickart objects.

\begin{defn}[{\cite[Definition~9.4]{CK}}] \rm Let $M$ and $N$ be objects of an abelian category $\mathcal{A}$. Then:
\begin{enumerate}
\item $N$ is called \emph{$M$-$\mathcal{K}$-nonsingular} if for any morphism $f:M\to N$ in $\mathcal{A}$, 
${\rm Ker}(f)$ essential in $M$ implies $f=0$. 
\item $M$ is called \emph{$N$-$\mathcal{T}$-nonsingular} if for any morphism $f:M\to N$ in $\mathcal{A}$, 
${\rm Im}(f)$ superfluous in $N$ implies $f=0$. 
\end{enumerate}
\end{defn}

\begin{theo} \label{t:nonsing} Let $M$ and $N$ be objects of an abelian category $\mathcal{A}$. Then:
\begin{enumerate}
\item $N$ is strongly $M$-CS-Rickart and $N$ is $M$-$\mathcal{K}$-nonsingular if and only if $N$ is strongly $M$-Rickart.
\item $N$ is dual strongly $M$-CS-Rickart and $M$ is $N$-$\mathcal{T}$-nonsingular if and only if $N$ is dual strongly $M$-Rickart.
\end{enumerate}
\end{theo}

\begin{proof} (1) Assume that $N$ is strongly $M$-CS-Rickart and $N$ is $M$-$\mathcal{K}$-nonsingular. Let $f:M \to N$ be a morphism in $\mathcal{A}$. Since $N$ is strongly $M$-CS-Rickart, there exists a fully invariant direct summand $U$ of $M$ such that ${\rm Ker}(f)$ is essential in $U$. Let $\sigma:U\to M$ be the inclusion monomorphism, which is fully  invariant, and consider the morphism $f\sigma:U\to N$. Since ${\rm Ker}(f\sigma)={\rm Ker}(f)$ is essential in $U$ and 
$N$ is $U$-$\mathcal{K}$-nonsingular \cite[Lemma~2.10]{CR1}, we have $f\sigma=0$. 
It follows that ${\rm Ker}(f)=U$. Hence ${\rm Ker}(f)$ is a fully invariant direct summand of $M$, which shows that $N$ is strongly $M$-Rickart.

Conversely, assume that $N$ is strongly $M$-Rickart. Then $N$ is strongly $M$-CS-Rickart. Now let $f:M\to N$ be a morphism in $\mathcal{A}$ such that ${\rm Ker}(f)$ is essential in $M$. But since $N$ is strongly $M$-Rickart, ${\rm Ker}(f)$ is a fully invariant direct summand of $M$. Hence ${\rm Ker}(f)=M$, and thus $f=0$.
Therefore, $N$ is $M$-$\mathcal{K}$-nonsingular. 
\end{proof}

Let $M$ and $N$ be objects of an abelian category $\mathcal{A}$. Following \cite{CK} and \cite[p.~220]{NZ},  $M$ is called \emph{direct $N$-injective} if every subobject of $N$ isomorphic to a direct summand of $M$ is a direct summand of $N$. Dually, $N$ is called \emph{direct $M$-projective} if for every factor object $M/K$ isomorphic to a direct summand of $N$, $K$ is a direct summand of $M$. For $M=N$ we obtain the concepts of \emph{direct injective} object (or \emph{$C_2$-object}) and \emph{direct projective} object (or \emph{$D_2$-object}) respectively.
Also, recall that $N$ is called \emph{(strongly) $M$-regular} if $N$ is (strongly) $M$-Rickart and dual (strongly) $M$-Rickart \cite{CK,CO2,DNTD}. Now we may characterize when a (strongly) $M$-CS-Rickart or dual (strongly) $M$-CS-Rickart object is (strongly) $M$-regular in terms of the above non-singularity conditions.

\begin{coll} Let $M$ and $N$ be objects of an abelian category $\mathcal{A}$. 
\begin{enumerate}
\item Assume that $N$ is (strongly) $M$-CS-Rickart. Then   $N$ is (strongly) $M$-regular if and only if $N$ is $M$-$\mathcal{K}$-nonsingular and $M$ is direct $N$-injective.
\item Assume that $N$ is dual (strongly) $M$-CS-Rickart. Then $N$ is (strongly) $M$-regular if and only if 
$M$ is $N$-$\mathcal{T}$-nonsingular and $N$ is direct $M$-projective.
\end{enumerate}
\end{coll}

\begin{proof} This follows by \cite[Theorem~5.3]{CK}, \cite[Theorem~2.10]{CO2},  \cite[Theorem~2.11]{CR1} and Theorem \ref{t:nonsing}.
\end{proof}

\section{Direct summands of (dual) relative strongly CS-Rickart objects}

We begin with a result showing that the (dual) strongly relative CS-Rickart property is preserved by direct summands.

\begin{theo} \label{t:sdr1} Let $r:M\to M'$ be an epimorphism and $s:N'\to N$ a monomorphism in an abelian category $\mathcal{A}$. 
\begin{enumerate} 
\item If $r$ is a retraction and $N$ is strongly $M$-CS-Rickart, then $N'$ is strongly $M'$-CS-Rickart.
\item If $s$ is a section and $N$ is dual strongly $M$-CS-Rickart, then $N'$ is dual strongly $M'$-CS-Rickart.
\end{enumerate}
\end{theo}

\begin{proof}
(1) Let $f:M'\to N'$ be a morphism in $\mathcal{A}$. Consider the morphism $sfr: M\to N$, and denote $L={\rm Ker}(fr)={\rm Ker}(sfr)$ and $l={\rm ker}(fr):L\to M$. Since $N$ is strongly $M$-CS-Rickart, 
there exist an essential monomorphism 
$e:L\to U$ and a fully invariant section $u:U\to M$ such that $l=ue$. Then by the proof of \cite[Theorem~2.6]{CR1} we have the following induced commutative diagram:
$$\SelectTips{cm}{}
\xymatrix{
 & 0 \ar[d] & 0 \ar[d] & 0\ar[d] & \\
 & X \ar@{=}[r] \ar[d]_i & X \ar@{=}[r] \ar[d]^{j} & X \ar[d]^{uj} \\ 
0 \ar[r] & L \ar[r]^-{e} \ar@{-->}[d]_p & U \ar[r]^-{u} \ar@{-->}[d]^-{q} & M \ar[r]^-{fr} \ar[d]^r & N' \ar@{=}[d] \ar[r]^s & N \ar@{=}[d] \\
0 \ar[r] & K \ar@{-->}[r]_-{k} \ar[d] & V \ar@{-->}[r]_-{v} \ar[d] & M' \ar[r]_f \ar[d] & N' \ar[r]_s & N \\
 & 0 & 0 & 0 &
}$$
with exact columns, where $k:K\to V$ is an essential monomorphism, $v:V\to M$ is a section and ${\rm ker}(f)=vk$. By \cite[Lemma~2.9]{CK} the square $UMM'V$ is a pushout, whence ${\rm coker}(v)r={\rm coker}(u)$. Since ${\rm coker}(u)$ is a fully coinvariant retraction, so is ${\rm coker}(v)$ by the dual of \cite[Proposition~2.8]{CKT}. Hence $v$ is a fully invariant section. Thus, $N'$ is strongly $M'$-CS-Rickart.
\end{proof}

\begin{coll} \label{c:sdr5} Let $M$ and $N$ be objects of an abelian category $\mathcal{A}$, 
$M'$ a direct summand of $M$ and $N'$ a direct summand of $N$. 
\begin{enumerate} 
\item If $N$ is strongly $M$-CS-Rickart, then $N'$ is strongly $M'$-CS-Rickart.
\item If $N$ is dual strongly $M$-CS-Rickart, then $N'$ is dual strongly $M'$-CS-Rickart.
\end{enumerate}
\end{coll}

Recall that an object $M$ of an abelian category $\mathcal{A}$ has \emph{SIP} (\emph{SSIP}) if for any family of (two) subobjects of $M$, their intersection is a direct summand, and it has \emph{SSP} (\emph{SSSP}) if for any family of (two) subobjects of $M$, their sum is a direct summand (e.g., see \cite{CK}). Also, $M$ is called \emph{SSIP-extending} (\emph{SIP-extending}) if for any family of (two) subobjects of $M$ which are essential in direct summands of $M$, their intersection is essential in a direct summand of $M$, and \emph{SSSP-lifting} (\emph{SSP-lifting}) if for any family of (two) subobjects of $M$ which lie above direct summands of $M$, their sum lies above a direct summand of $M$ \cite{CR1} (also see  \cite{AI14,AI15,ANQ,Karab} for modules). Next we consider some strict versions of SSIP-extending and SIP-extending objects, and their duals, SSSP-lifting and SSP-lifting objects in abelian categories. 

\begin{defn} \rm An object $M$ of an abelian category $\mathcal{A}$ is called: 
\begin{enumerate}
\item \emph{strictly SSIP-extending} (\emph{strictly SIP-extending}) if for any family of (two) subobjects of $M$ 
which are essential in direct summands of $M$, their intersection is essential in a fully invariant direct summand of $M$.
\item \emph{strictly SSSP-lifting} (\emph{strictly SSP-lifting}) if for any family of (two) subobjects of $M$ 
which lie above direct summands of $M$, their sum lies above a fully invariant direct summand of $M$.
\end{enumerate}
\end{defn}

The following lemma will be useful, and its proof is similar to that of \cite[Lemma~3.4]{CR1}.

\begin{lemm} \label{l:SIPSSP} Let $M$ be an object of an abelian category $\A$. Then:
\begin{enumerate}
\item $M$ is strictly SIP-extending if and only if the intersection of any two direct summands of $M$ is essential in a fully invariant direct summand of $M$. 
\item $M$ is strictly SSP-lifting if and only if the sum of any two direct summands of $M$ lies above a fully invariant direct summand of $M$.
\end{enumerate}
\end{lemm}

\begin{prop} \label{p:sdr3} Let $M$ and $N$ be objects of an abelian category $\mathcal{A}$. 
\begin{enumerate}
\item Assume that every direct summand of $M$ is isomorphic to a subobject of $N$, and $N$ is strongly $M$-CS-Rickart. Then $M$ is strictly SIP-extending.
\item Assume that every direct summand of $N$ is isomorphic to a factor object of $M$, and $N$ is dual strongly $M$-CS-Rickart. Then $N$ is strictly SSP-lifting.
\end{enumerate}
\end{prop}

\begin{proof} (1) Let $M_1$ and $M_2$ be direct summands of $M$. Then $M_1\cong N_1$ and $M_2\cong N_2$ for some subobjects $N_1$ and $N_2$ of $N$.
Since $N$ is strongly $M$-CS-Rickart, $N_1+N_2$ is strongly $M$-CS-Rickart by Theorem~\ref{t:sdr1}. Since 
$M_2$ is a direct summand of $M$, the following canonical short exact sequence splits $$0\to M_2\to M_1+M_2\to (M_1+M_2)/M_2\to 0.$$ Then there is an induced composite monomorphism 
$$N_1/(N_1\cap N_2)\cong (N_1+N_2)/N_2\cong (M_1+M_2)/M_2\to M_1+M_2\cong N_1+N_2\to N.$$ By Theorem~\ref{t:sdr1},  $N_1/(N_1\cap N_2)$ is strongly $M$-CS-Rickart, and furthermore, $N_1/(N_1\cap N_2)$ is strongly $M_1$-CS-Rickart, since $M_1$ is a direct summand of $M$. Consider the induced short exact sequence 
$$0\to M_1\cap M_2\to M_1\to N_1/(N_1\cap N_2)\to 0.$$ Since $N_1/(N_1\cap N_2)$ is $M_1$-CS-Rickart, 
$M_1\cap M_2$ is essential in a fully invariant direct summand $M_3$ of $M_1$, and so $M_1\cap M_2$ is essential in a fully invariant direct summand $M_3$ of  $M$. Hence $M$ is strictly SIP-extending by Lemma \ref{l:SIPSSP}. 
\end{proof}

\begin{coll} \label{c:sdr4} Let $\mathcal{A}$ be an abelian category. 
\begin{enumerate}
\item Every strongly self-CS-Rickart object of $\mathcal{A}$ is strictly SIP-extending.
\item Every dual strongly self-CS-Rickart object of $\mathcal{A}$ is strictly SSP-lifting.
\end{enumerate}
\end{coll}

\begin{lemm} \label{l:AB} Let $A$ and $B$ be objects of an abelian category $\A$.  
\begin{enumerate}
\item If $A\oplus B$ is strictly SIP-extending, then $B$ is strongly $A$-CS-Rickart.
\item If $A\oplus B$ is strictly SSP-lifting, then $B$ is dual strongly $A$-CS-Rickart.
\end{enumerate} 
\end{lemm}

\begin{proof} This is similar to the proof of \cite[Lemma~3.7]{CR1}.
\end{proof}

\begin{ex} \label{e:SIP} \rm Consider the $\mathbb{Z}$-module $M=\mathbb{Z}_2\oplus \mathbb{Z}_{16}$. Since $\mathbb{Z}_2$ and $\mathbb{Z}_{16}$ are both uniform and hollow, $\mathbb{Z}_{16}$ is strongly $\mathbb{Z}_2$-CS-Rickart and dual strongly $\mathbb{Z}_2$-CS-Rickart, and $\mathbb{Z}_2$ is strongly $\mathbb{Z}_{16}$-CS-Rickart and dual strongly $\mathbb{Z}_{16}$-CS-Rickart.

The direct summands of $M$ are the following: \[\{(0,0)\}, M, A=\langle (1,1)\rangle, B=\langle (0,1)\rangle, C=\langle(1,0)\rangle, D=\langle(1,8)\rangle\] generated by the specified elements of $M$. The fully invariant direct summands of $M$ are only $\{(0,0)\}$ and $M$. Since $A\cap B$ is not essential in $M$ and $A+B$ is not superfluous in $M$, $M$ is neither strictly SIP-extending, nor strictly SSP-lifting by Lemma \ref{l:SIPSSP}. This also shows that $M$ is neither strongly self-CS-Rickart, nor dual strongly self-CS-Rickart by Corollary \ref{c:sdr4}.

On the other hand, $M$ is SIP-extending by \cite[Example~3.9]{CR1}. Also, we have $M=C\oplus B$, where $C\subseteq C+D$ and $B\cap (C+D)$ is  superfluous in $B$. Then $C+D$ lies above the direct summand $C$ of $M$ \cite[41.11]{Wis}. Finally, using \cite[Lemma~3.4]{CR1}, $M$ is SSP-lifting.  
\end{ex}

\section{(Co)products of (dual) relative strongly CS-Rickart objects}

Our main result on (co)products of (dual) strongly relative CS-Rickart objects is the following one.

\begin{theo} \label{t:sdr2} Let $\mathcal{A}$ be an abelian category. 
\begin{enumerate} \item Let $M$, $N_1$ and $N_2$ be objects of $\mathcal{A}$
such that $N_1$ and $N_2$ are strongly $M$-CS-Rickart. Then $N_1\oplus N_2$ is strongly $M$-CS-Rickart.  
\item Let $M_1$, $M_2$ and $N$ be objects of $\mathcal{A}$ such that $N$ is dual strongly $M_1$-CS-Rickart and dual strongly $M_2$-CS-Rickart. Then $N$ is dual strongly $M_1\oplus M_2$-CS-Rickart. 
\end{enumerate}
\end{theo}

\begin{proof} (1) Denote $N=N_1\oplus N_2$, and let $f:M\to N$ be a morphism in $\mathcal{A}$. 
As in the proof of \cite[Theorem~4.1]{CR1}, we have the following induced commutative diagrams
$$\SelectTips{cm}{}
\xymatrix{
 & 0 \ar[d] & 0 \ar[d] & \\
 & K_1 \ar@{=}[r] \ar[d]_{i_1} & K_1 \ar[d]^{i_2} & \\ 
0 \ar[r] & K_2 \ar[r]^-{k_2} \ar[d]_{f_1} & M \ar[r]^-{f_2} \ar[d]^f & N_2 \ar@{=}[d] \ar[r] & 0 \\
0 \ar[r] & N_1 \ar[r]_{j_1} & N \ar[r]_{p_2} & N_2 \ar[r] & 0 
}\qquad \SelectTips{cm}{}
\xymatrix{
&&&\\ 
0 \ar[r] & K_1 \ar[r]^-{i_1} \ar[d]_{e_1} & K_2 \ar[r]^{f_1} \ar[d]^{e_2} & N_1 \ar@{=}[d] \\ 
0 \ar[r] & K \ar[r]_k & U \ar[r]_-{p_1fu} & N_1 \\
}$$
with exact rows and columns, where $j_1:N_1\to N$ is the canonical injection, $p_2:N\to N_2$ is the canonical projection, $e_1:K_1\to K$ and $e_2:K_2\to U$ are essential monomorphisms and $u:U\to M$ is a fully invariant section such that $k_2=ue_2$.  

Since $N_1$ is $U$-CS-Rickart by Theorem \ref{t:sdr1},  
there exist an essential monomorphism $e:K\to V$ and a fully invariant section $v:V\to U$ such that $k=ve$. 
It follows that $${\rm ker}(f)=i_2=k_2i_1=ue_2i_1=uke_1=uvee_1,$$ 
where $ee_1:K_1\to V$ is an essential monomorphism and  $uv:V\to M$ is a fully invariant section. This shows that $N$ is strongly $M$-CS-Rickart.
\end{proof}

Now we may immediately deduce the next theorem.

\begin{theo} \label{t:pr1} Let $\mathcal{A}$ be an abelian category.
\begin{enumerate} 
\item Let $M$ and $N_1,\dots,N_n$ be objects of $\mathcal{A}$. Then $\bigoplus_{i=1}^n N_i$ is strongly $M$-CS-Rickart if and only
if $N_i$ is strongly $M$-CS-Rickart for every $i\in \{1,\dots,n\}$.
\item Let $M_1,\dots,M_n$ and $N$ be objects of $\mathcal{A}$. Then $N$ is dual strongly $\bigoplus_{i=1}^n M_i$-CS-Rickart if and
only if $N$ is dual strongly $M_i$-CS-Rickart for every $i\in \{1,\dots,n\}$.
\end{enumerate}
\end{theo}

\begin{proof} This follows by Corollary~\ref{c:sdr5} and Theorem~\ref{t:sdr2}.
\end{proof}

In order to get some similar results in case of arbitrary (co)products, we need to impose some finiteness conditions.

\begin{coll} \label{c:pr2} Let $\mathcal{A}$ be an abelian category.
\begin{enumerate} \item Assume that $\mathcal{A}$ has coproducts, let $M$ be a finitely generated object of
$\mathcal{A}$, and let $(N_i)_{i\in I}$ be a family of objects of $\mathcal{A}$. Then $\bigoplus_{i\in I}
N_i$ is strongly $M$-CS-Rickart if and only if $N_i$ is strongly $M$-CS-Rickart for every $i\in I$.

\item Assume that $\mathcal{A}$ has products, let $N$ be a finitely cogenerated object of $\mathcal{A}$, and let
$(M_i)_{i\in I}$ be a family of objects of $\mathcal{A}$. Then $N$ is dual strongly $\prod_{i\in I} M_i$-CS-Rickart if and only if $N$ is dual strongly $M_i$-CS-Rickart for every $i\in I$.
\end{enumerate}
\end{coll}

\begin{proof} (1) First, let $f:M\to
\bigoplus_{i\in I} N_i$ be a morphism in $\mathcal{A}$. Since $M$ is finitely generated, we may write
$f=jf'$ for some morphism $f':M\to \bigoplus_{i\in F} N_i$ and inclusion morphism $j:\bigoplus_{i\in F} N_i\to
\bigoplus_{i\in I} N_i$, where $F$ is a finite subset of $I$. By Theorem~\ref{t:pr1}, $\bigoplus_{i\in F} N_i$ is
strongy $M$-CS-Rickart. Then there exist an essential monomorphism $e:{\rm Ker}(f')\to U$ and a section $u:U\to M$ such that ${\rm ker}(f')=ue$. 
But ${\rm ker}(f)={\rm ker}(f')$. Hence $\bigoplus_{i\in I} N_i$ is strongly $M$-CS-Rickart. 

The converse follows by Corollary~\ref{c:sdr5}.
\end{proof}

A necessary condition for an infinite (co)product of objects to be a (dual) strongly self-CS-Rickart object is given in the following proposition.

\begin{prop} \label{p:pr3} Let $(M_i)_{i\in I}$ be a family of objects of an abelian category $\mathcal{A}$.
\begin{enumerate} 
\item If $\prod_{i\in I} M_i$ is a strongly self-CS-Rickart object, then $M_i$ is strongly $M_j$-CS-Rickart for every $i,j\in I$.
\item If $\bigoplus_{i\in I} M_i$ is a dual self-CS-Rickart object, then $M_i$ is dual $M_j$-CS-Rickart for every
$i,j\in I$.
\end{enumerate}
\end{prop}

\begin{proof} This follows by Theorem~\ref{t:sdr1}.
\end{proof}

\begin{ex} \rm Consider the $\mathbb{Z}$-module $M=\mathbb{Z}_2\oplus \mathbb{Z}_{16}$. 
We have seen in Example \ref{e:SIP} that $\mathbb{Z}_{16}$ is (dual) strongly $\mathbb{Z}_2$-CS-Rickart. Also, $\mathbb{Z}_2$ is clearly (dual) strongly $\mathbb{Z}_{16}$-CS-Rickart. Also, both $\mathbb{Z}_2$ and $\mathbb{Z}_{16}$ 
are (dual) strongly self-CS-Rickart, because they are uniform (hollow). On the other hand, $M=\mathbb{Z}_2\oplus \mathbb{Z}_{16}$ is not (dual) strongly self-CS-Rickart, because it is not (dual) self-CS-Rickart \cite[Example~3.9]{CR1}. Hence the converse of Proposition \ref{p:pr3} does not hold.
\end{ex}

\begin{theo} Let $\mathcal{A}$ be an abelian category.
\begin{enumerate}
\item Let $M$ be a strictly SSIP-extending object of $\mathcal{A}$, and let $(N_i)_{i\in I}$ be a family of objects of $\mathcal{A}$ having a product. Then $\prod_{i\in I} N_i$ is
strongly $M$-CS-Rickart if and only if $N_i$ is strongly $M$-CS-Rickart for every $i\in I$.
\item Let $(M_i)_{i\in I}$ be a family of objects of $\mathcal{A}$ having a coproduct, and let $N$ be a strictly SSSP-lifting object of $\mathcal{A}$. Then $N$ is dual strongly $\bigoplus_{i\in I} M_i$-CS-Rickart if and only if $N$ is dual strongly $M_i$-CS-Rickart for every $i\in I$.
\end{enumerate}
\end{theo}

\begin{proof} (1) First, assume that $N_i$ is strongly $M$-CS-Rickart for every $i\in I$. Let $f:M\to \prod_{i\in I} N_i$ be a morphism in $\mathcal{A}$. For every $i\in I$, denote by $p_i:\prod_{i\in I} N_i\to N_i$ the canonical projection and $f_i=p_if:M\to N_i$. Since $N_i$ is strongly $M$-CS-Rickart, ${\rm Ker}(f_i)$ is essential in a fully invariant direct summand of $M$ 
for every $i\in I$. Since $M$ is strictly SSIP-extending, 
${\rm Ker}(f)=\bigcap_{i\in I} {\rm Ker}(f_i)$ is essential in a fully invariant direct summand of $M$. 
Hence $\prod_{i\in I} N_i$ is strongly $M$-CS-Rickart.

The converse follows by Theorem \ref{t:sdr1}. 
\end{proof}

The coproduct of two (dual) strongly self-CS-Rickart objects need not be (dual) strongly self-CS-Rickart, 
as the following example shows.

\begin{ex} \label{ec1} \rm (i) $\mathbb{Z}_2$ is a (dual) strongly self-CS-Rickart $\mathbb{Z}$-module, but $\mathbb{Z}_2\oplus \mathbb{Z}_2$ is not a (dual) strongly self-CS-Rickart $\mathbb{Z}$-module.

(ii) Consider the ring $R=\begin{pmatrix}{K}&K\\0&K \end{pmatrix}$ for some field $K$. The right $R$-modules $M_1=\begin{pmatrix}K&K\\0&0 \end{pmatrix}$ 
and $M_2=\begin{pmatrix}0&0\\0&K \end{pmatrix}$
are (dual) self-Rickart and ${\rm End}(M_1)\cong {\rm End}(M_2)\cong K$, and so $M_1$ and  $M_2$ are (dual) strongly self-CS-Rickart by Proposition \ref{st1}. But 
$R=M_1\oplus M_2$ is not  (dual) strongly self-CS-Rickart again by Proposition \ref{st1}.     
\end{ex}

But we may give the following theorem.

\begin{theo} \label{t:pstr4} Let $\mathcal{A}$ be an abelian category, and let $M=\bigoplus_{i\in I}M_i$ be a direct sum decomposition in $\mathcal{A}$. 
\begin{enumerate}
\item Then $M$ is strongly self-CS-Rickart if and only if $M_i$ is strongly self-CS-Rickart for each $i\in I$ and $\Hom_{\mathcal{A}}(M_i,M_j)=0$ for every $i,j\in I$ with $i\neq j$.
\item Assume that $I$ is finite. Then $M$ is  dual strongly self-CS-Rickart if and only if $M_i$ is dual strongly self-CS-Rickart for each $i\in I$ and $\Hom_{\mathcal{A}}(M_i,M_j)=0$ for every $i,j\in I$ with $i\neq j$.
\end{enumerate}
\end{theo}

\begin{proof} (1) Assume that $M$ is strongly self-CS-Rickart. Then $M$ is weak duo by Corollary \ref{st00}, hence its direct summand $M_i$ is fully invariant for each $i\in I$. It follows that ${\rm Hom}_{\mathcal{A}}(M_i,M_j)=0$ for every $i,j\in I$ with $i\neq j$, and $M_i$ is strongly self-CS-Rickart for every $i\in I$ by Corollary \ref{c:sdr5}.

Conversely, assume that $M_i$ is strongly self-CS-Rickart for every $i\in I$. Let $f:M\to M$ be a morphism in $\mathcal{A}$. Then $f=\bigoplus_{i\in I}f_i$ for some morphisms $f_i:M_i\to M_i$ ($i\in I$) and $K={\rm Ker}(f)=\bigoplus_{i\in I}{\rm Ker}(f_i)$.
Now let $i\in I$. Since $M_i$ is strongly self-CS-Rickart, there are an essential monomorphism $u_i:K_i\to U_i$ and a fully invariant section $s_i:U_i\to M_i$ such that ${\rm ker}(f_i)=s_iu_i$. Then we may write $k=su$, where 
$u=\bigoplus_{i\in I}u_i:\bigoplus_{i\in I}K_i\to \bigoplus_{i\in I}U_i$ is an essential monomorphism 
and $s=\bigoplus_{i\in I}s_i:\bigoplus_{i\in I}U_i\to \bigoplus_{i\in I}M_i$ is a section. 
For each $i\in I$, there exists $p_i:M_i\to U_i$ such that $p_is_i=1_{U_i}$ and $e_i=s_ip_i$ is an idempotent in the abelian ring ${\rm End}_{\mathcal{A}}(M)$. Denote $p=\bigoplus_{i\in I}p_i:\bigoplus_{i\in I}M_i\to \bigoplus_{i\in I}U_i$. Then $ps=1_{\bigoplus_{i\in I}U_i}$ and $e=sp=\bigoplus_{i\in I}s_ip_i=\bigoplus_{i\in I}e_i$ is a central idempotent of the ring ${\rm End}_{\mathcal{A}}(M)$. Now by \cite[Lemma~2.13]{CO1}, $s$ is a fully invariant section. This shows that $M$ is strongly self-CS-Rickart.

(2) This part uses images instead of kernels. We point out that here we need the index set $I$ to be finite, because in general a direct sum of superfluous epimorphisms is not  superfluous.
\end{proof}

We end this section with some results on the structure of (dual) strongly self-CS-Rickart modules over a Dedekind domain.

\begin{coll} \label{c1:dede}
Let $R$ be a Dedekind domain with quotient field $K$, and let $M$ be a non-zero $R$-module. 

\begin{enumerate}[(i)] 
\item Assume that $M$ is torsion. The following are equivalent:
\begin{enumerate}[(1)]
\item $M$ is strongly self-CS-Rickart.
\item $M$ is dual strongly self-CS-Rickart.
\item $M$ is weak duo. 
\item $M\cong \bigoplus_{i\in I}M_i$, where for each $i\in I$, either $M_i\cong E(R/P_i)$ or $M_i\cong R/P_i^{n_i}$ for some distinct maximal ideals $P_i$ of $R$ and positive integers $n_i$. 
\end{enumerate}

\item Assume that $M$ is finitely generated. 
\begin{enumerate}[(1)]
\item The following are equivalent:
\begin{enumerate}[(a)] 
\item $M$ is strongly self-CS-Rickart.
\item $M$ is weak duo. 
\item $M\cong J$ for some ideal $J$ of $R$ or $M\cong \bigoplus_{i=1}^kR/P_i^{n_i}$ for some distinct maximal ideals $P_1,\dots,P_k$ of $R$ and positive integers $n_1,\dots,n_k$.
\end{enumerate}
\item The following are equivalent:
\begin{enumerate}[(a)]
\item $M$ is dual strongly self-CS-Rickart.
\item $M\cong \bigoplus_{i=1}^kR/P_i^{n_i}$ for some distinct maximal ideals $P_i$ of $R$ and positive integers $n_1,\dots,n_k$. 
\end{enumerate}
\end{enumerate}

\item Assume that $M$ is injective. Then the following are equivalent:
\begin{enumerate}[(1)]
\item $M$ is strongly self-CS-Rickart.
\item $M$ is dual strongly self-CS-Rickart.
\item $M\cong K$ or $M\cong \bigoplus_{i\in I} E(R/P_i)$ for some distinct maximal ideals $P_i$ of $R$. 
\end{enumerate}
\end{enumerate}
\end{coll}

\begin{proof} (i) The implications (1)$\Rightarrow$(3) and (2)$\Rightarrow$(3) follow by Corollary \ref{st00}, while 
the equivalence (3)$\Leftrightarrow$(4) is true by \cite[Theorem~3.10]{OHS}.

(4)$\Rightarrow$(1) and (4)$\Rightarrow$(2) Under the hypothesis (4), $M$ is extending \cite[Corollary~23]{KM} and lifting \cite[Propositions~A.7,A.8]{MM}, hence $M$ is self-CS-Rickart and dual self-CS-Rickart. But $M$ is also weak duo by the equivalence (3)$\Leftrightarrow$(4), and thus $M$ is strongly self-CS-Rickart and dual strongly self-CS-Rickart.

(ii) The implication (a)$\Rightarrow$(b) follows by Corollary \ref{st00}, while 
the equivalence (b)$\Leftrightarrow$(c) is true by \cite[Corollary~3.11]{OHS}.

(c)$\Rightarrow$(a) If $M\cong J$ for some ideal $J$ of $R$, then $M$ is strongly self-Rickart \cite[Corollary~3.8]{CO1}, and thus, it is strongly self-CS-Rickart. But $M$ is also weak duo by the equivalence (b)$\Leftrightarrow$(c), and thus $M$ is strongly self-CS-Rickart.

(2) Assume first $M$ is dual strongly self-CS-Rickart. Then $M$ is weak duo by Corollary \ref{st00}, and so $M\cong J$ for some ideal $J$ of $R$ or $M\cong \bigoplus_{i=1}^kR/P_i^{n_i}$ for some distinct maximal ideals $P_1,\dots,P_k$ of $R$ and positive integers $n_1,\dots,n_k$ \cite[Corollary~3.11]{OHS}. Since $M\cong J$ is uniform, it is indecomposable, hence it cannot be dual self-CS-Rickart by \cite[Proposition~4.15]{Tribak}. Then it is not dual strongly self-CS-Rickart by Corollary \ref{st01}. Hence $M$ has the required form.

Conversely, assume that $M\cong \bigoplus_{i=1}^kR/P_i^{n_i}$ for some distinct maximal ideals $P_1,\dots,P_k$ of $R$ and positive integers $n_1,\dots,n_k$. Then $M$ is lifting \cite[Propositions~A.7,A.8]{MM}, hence $M$ is dual self-CS-Rickart. But $M$ is also weak duo by \cite[Corollary~3.11]{OHS}, and thus $M$ is dual strongly self-CS-Rickart. 

(iii) Every injective module $M$ over a Dedekind domain $R$ is of the form 
$M=\bigoplus_{i\in I}M_i$, where $M_i$ is either a vector space over $K$ or $E(R/P_i)$ for some prime ideal $P_i$ of $R$ \cite[Theorem~7]{K52}. 

(1)$\Rightarrow$(3) and (2)$\Rightarrow$(3) If $M$ is (dual) strongly self-CS-Rickart, then $M$ is weak duo by Corollary \ref{st00}, and thus $M$ cannot contain copies of the same $M_i$ \cite[Theorem~2.7]{OHS}. Since ${\rm Hom}_R(E(R),E(R/P))\neq 0$ and $M$ is weak duo, $M$ cannot contain both $K$ and $E(R/P)$ for some prime ideal $P$ of $R$ \cite[Theorem~2.7]{OHS}. Since each such $E(R/P)$ is torsion weak duo,  
$P$ must be a maximal ideal of $R$ by (i).
Note that ${\rm Hom}_R(E(R/P),E(R/Q))=0$ for every distinct maximal ideals $P$ and $Q$ of $R$ \cite[Proposition~4.21]{SV}. Hence $M\cong K$ or $M\cong \bigoplus_{i\in I} E(R/P_i)$ for some distinct maximal ideals $P_i$ of $R$.

(3)$\Rightarrow$(1) and (3)$\Rightarrow$(2) 
If $M\cong K$, then $M$ is clearly (dual) self-CS-Rickart. Also, if $M\cong \bigoplus_{i\in I} E(R/P_i)$ for some distinct maximal ideals $P_i$ of $R$, then $M$ is (dual) self-CS-Rickart by (i). 
\end{proof}

\begin{coll} \label{c1:abgr}
Let $G$ be a non-zero abelian group. 

\begin{enumerate}[(i)] 
\item Assume that $G$ is torsion. The following are equivalent: 
\begin{enumerate}[(1)]
\item $G$ is strongly self-CS-Rickart.
\item $G$ is dual strongly self-CS-Rickart.
\item $G$ is weak duo. 
\item $G\cong \bigoplus_{i\in I}G_i$, where for each $i\in I$, either $G_i\cong \mathbb{Z}_{p_i^{\infty}}$ or $G_i\cong \mathbb{Z}_{p_i^{n_i}}$ for some distinct primes $p_i$ and positive integers $n_i$.
\end{enumerate}
 
\item Assume that $G$ is finitely generated. 
\begin{enumerate}[(1)] 
\item The following are equivalent:
\begin{enumerate}[(a)]
\item $G$ is strongly self-CS-Rickart.
\item $G$ is weak duo.
\item $G\cong n\mathbb{Z}$ for some positive integer $n$ or $G\cong \bigoplus_{i=1}^k \mathbb{Z}_{p_i^{n_i}}$ for some distinct primes $p_i$ and positive integers $n_i$.
\end{enumerate}
\item The following are equivalent:
\begin{enumerate}[(a)]
\item $G$ is dual strongly self-CS-Rickart.
\item $G\cong \bigoplus_{i=1}^k \mathbb{Z}_{p_i^{n_i}}$ for some distinct primes $p_i$ and positive integers $n_1,\dots,n_k$.
\end{enumerate}
\end{enumerate}

\item Assume that $G$ is injective. Then the following are equivalent:
\begin{enumerate}[(1)]
\item $G$ is strongly self-CS-Rickart.
\item $G$ is dual strongly self-CS-Rickart.
\item $G\cong \mathbb{Q}$ or $G\cong \bigoplus_{i\in I} \mathbb{Z}_{p_i^{\infty}}$ for some primes $p_i$. 
\end{enumerate}
\end{enumerate}
\end{coll}

\begin{ex} \label{e1:abgr} \rm (i) The abelian group $\mathbb{Z}_p\oplus \mathbb{Z}_p$ (for some prime $p$) is both self-CS-Rickart and dual self-CS-Rickart, being semisimple. But it is neither strongly self-CS-Rickart, nor dual strongly self-CS-Rickart by Corollary~\ref{c1:abgr}.

(ii) The abelian group $\mathbb{Z}_p\oplus \mathbb{Q}$ (for some prime $p$) is strongly self-CS-Rickart and dual strongly self-CS-Rickart by Corollary~\ref{c1:abgr}, but neither strongly extending \cite[Example~2.9]{EKD}, nor strongly lifting \cite[Example~3.9]{Wang}.
\end{ex}

\section{Classes all of whose objects are (dual) strongly CS-Rickart}

We may use the above results to obtain characterizations of some classes all of whose objects are (dual) strongly self-CS-Rickart. Following the corresponding module-theoretic notions, an object $M$ of an abelian category $\mathcal{A}$ is called \emph{(dual) square-free} if whenever a subobject (factor object) of $M$ is isomorphic to $N\oplus N$ for some object $N$ of $\mathcal{A}$, one has $N=0$ (e.g., see \cite{KKKS}). Note that a semisimple object of $\mathcal{A}$ is weak duo if and only if it is (dual) square-free.

\begin{theo} \label{t:extlif} Let $\A$ be an abelian category. 
\begin{enumerate}
\item Assume that $\A$ has enough injectives. Let $\mathcal{C}$ be a class of objects of $\A$ which is closed under binary direct sums and contains all injective objects of $\A$. Consider the following conditions:
\begin{enumerate}[(i)]
\item Every object of $\C$ is strongly extending.
\item Every object of $\C$ is strongly self-CS-Rickart.
\item Every object of $\C$ is weak duo injective. 
\end{enumerate}
Then (i)$\Leftrightarrow$(ii)$\Rightarrow$(iii). If every object of $\A$ has an injective envelope, then all three conditions are equivalent.
\item Assume that $\A$ has enough projectives. Let $\mathcal{C}$ be a class of objects of $\A$
which is closed under binary direct sums and contains all projective objects of $\A$. Consider the following conditions:
\begin{enumerate}[(i)]
\item Every object of $\C$ is strongly lifting.
\item Every object of $\C$ is dual strongly self-CS-Rickart.
\item Every object of $\C$ is weak duo projective.
\end{enumerate}
Then (i)$\Leftrightarrow$(ii)$\Rightarrow$(iii). If every object of $\A$ has a projective cover (i.e., $\mathcal{A}$ is perfect), then all three conditions are equivalent.
\end{enumerate} 
\end{theo}

\begin{proof} (1) (i)$\Leftrightarrow$(ii) This follows by Corollary \ref{st00} and \cite[Theorem~5.1]{CR1}.

(ii)$\Rightarrow$(iii) Assume that every object of $\C$ is strongly self-CS-Rickart. Then every object of $\C$ is weak duo by Corollary \ref{st00}, which implies that every object of $\C$ has SIP by \cite[Corollary~2.2]{OHS}, whose proof works in abelian categories. Now let $M$ be an object of $\mathcal{C}$. Consider some monomorphisms $i_1:M\to E_1$ and $i_2:E_1/M\to E_2$ for some injective objects $E_1$, $E_2$ of $\mathcal{A}$, and the the natural epimorphism $p_1:E_1\to E_1/M$. 
Since $E_1\oplus E_2\in \C$, it has SIP. Then $E_2$ is $E_1$-Rickart by \cite[Proposition~3.7]{CK}. Hence for the morphism $f=i_2p_1:E_1\to E_2$, $M={\rm Ker}(f)$ is a direct summand of $E_1$. Hence $M$ is weak duo injective.

(iii)$\Rightarrow$(i) Assume that every object of $\C$ is weak duo injective. Since every object of $\A$ has an injective envelope, every injective object is extending by \cite[Corollary~5.2]{CR1}. Also, the proof of Corollary \ref{st00} may be easily adapted in order to deduce that an object of an abelian category is strongly extending if and only if it is weak duo extending. Hence every object of $\C$ is strongly extending.  
\end{proof}

\begin{coll} \label{c:extlifgen} Let $\A$ be an abelian category. 
\begin{enumerate}
\item Assume that $\A$ has enough injectives. Consider the following conditions:
\begin{enumerate}[(i)]
\item Every (injective) object of $\A$ is strongly extending.
\item Every (injective) object of $\A$ is strongly self-CS-Rickart.
\item Every (injective) object of $\A$ is weak duo injective. 
\end{enumerate}
Then (i)$\Leftrightarrow$(ii)$\Rightarrow$(iii). 
If every object of $\A$ has an injective envelope, then all three conditions are equivalent.
\item Assume that $\A$ has enough projectives. Consider  the following conditions:
\begin{enumerate}[(i)]
\item Every (projective) object of $\A$ is strongly lifting.
\item Every (projective) object of $\A$ is dual strongly self-CS-Rickart.
\item Every (projective) object of $\A$ is weak duo projective.
\end{enumerate}
Then (i)$\Leftrightarrow$(ii)$\Rightarrow$(iii).
If every object of $\A$ has a projective cover (i.e., $\mathcal{A}$ is perfect), then all three conditions are equivalent.
\end{enumerate} 
\end{coll}

\begin{coll} The following are equivalent for a unitary ring $R$:
\begin{enumerate}[\indent (i)]
\item $R$ is square-free semisimple. 
\item Every right $R$-module is square-free semisimple.
\item Every right $R$-module is weak duo injective.
\item Every right $R$-module is strongly extending.
\item Every right $R$-module is strongly self-CS-Rickart.
\item Every right $R$-module is weak duo projective.
\item Every right $R$-module is strongly lifting.
\item Every right $R$-module is dual strongly self-CS-Rickart.
\end{enumerate}
\end{coll}

\begin{proof} This follows by Corollary \ref{c:extlifgen}, \cite[Lemma~5]{AN1} and the well-known equivalences: $R$ is semisimple if and only if every right $R$-module is semisimple if and only if every right $R$-module is injective if and only if every right $R$-module is projective.
\end{proof}

Recall that a Grothendieck category $\mathcal{A}$ is called a \emph{$V$-category} if every simple object is injective \cite{DNV}, and a \emph{regular} category if every object $B$ of $\mathcal{A}$ is regular in the sense that every short exact sequence of the form $0\to A\to B\to C\to 0$ is pure in $\mathcal{A}$ \cite[p.~313]{Wis}.

\begin{theo} \label{t:reg} Let $\A$ be a locally finitely generated Grothendieck category. 
\begin{enumerate}
\item The following are equivalent:
\begin{enumerate}[(i)]
\item Every finitely cogenerated object of $\A$ is weak duo semisimple. 
\item Every finitely cogenerated object of $\A$ is strongly self-CS-Rickart.
\item Every finitely cogenerated object of $\A$ is weak duo and every finitely cogenerated injective object of $\A$ is strongly self-CS-Rickart.
\end{enumerate}
\item The following are equivalent:
\begin{enumerate}[(i)]
\item Every finitely generated object of $\A$ is weak duo regular. 
\item Every finitely generated object of $\A$ is dual strongly self-CS-Rickart.
\item Every finitely generated object of $\A$ is weak duo and every finitely generated projective object of $\A$ is dual strongly self-CS-Rickart.
\end{enumerate}
\end{enumerate}
\end{theo}

\begin{proof} (1) (i)$\Rightarrow$(ii) If (i) holds, then every finitely cogenerated object of $\A$ is semisimple, and thus clearly self-CS-Rickart. But every finitely cogenerated object of $\A$ is also weak duo by hypothesis. Now the conclusion follows by Corollary \ref{st00}.

(ii)$\Rightarrow$(iii) This is obvious.

(iii)$\Rightarrow$(i) If (iii) holds, then every finitely cogenerated injective object of $\A$ has SIP by \cite[Corollary~2.2]{OHS}, whence we deduce that $\A$ is a $V$-category by \cite[Proposition~1.8]{Garcia}, both cited results having proofs which are valid in our setting. Then every finitely cogenerated object of $\A$ is semisimple \cite[23.1]{Wis} (also see \cite[Theorem~2.3]{DNV}). But every finitely cogenerated object of $\A$ is also weak duo by hypothesis.

(2) (i)$\Rightarrow$(ii) If (i) holds, then $\A$ is a regular category. Hence every finitely generated object of $\A$ is dual self-Rickart \cite[Theorem~4.4]{CK}, and thus dual self-CS-Rickart. But every finitely generated object of $\A$ is weak duo by hypothesis. Now the conclusion follows by Corollary \ref{st00}.

(ii)$\Rightarrow$(iii) This is obvious.

(iii)$\Rightarrow$(i) If (iii) holds, then every finitely generated projective object of $\A$ has SSP by \cite[Corollary~2.2]{OHS}, whence we deduce that $\A$ is a regular category by \cite[Proposition~1.8]{Garcia}, because the proofs of the cited results work in our setting. Then every finitely generated object of $\A$ is regular. But every finitely generated object of $\A$ is also weak duo by hypothesis.
\end{proof}

Recall that an abelian category $\A$ is called \emph{(semi)hereditary} if every (finitely generated) subobject of a projective object is projective, and \emph{(semi)cohereditary} if every (finitely cogenerated) factor object of an injective object is injective.

\begin{theo} \label{t:her} Let $\A$ be a locally finitely generated Grothendieck category. 
\begin{enumerate}
\item The following are equivalent:
\begin{enumerate}[(i)]
\item Every (finitely generated) subobject of a projective object of $\A$ is weak duo projective. 
\item Every (finitely generated) projective object of $\A$ is weak duo and every (finitely generated) projective object of $\A$ is strongly self-CS-Rickart.
\end{enumerate}
\item The following are equivalent:
\begin{enumerate}[(i)]
\item Every (finitely cogenerated) factor object of an injective object of $\A$ is weak duo injective. 
\item Every (finitely cogenerated) injective object of $\A$ is weak duo and every (finitely cogenerated) injective object of $\A$ is dual strongly self-CS-Rickart.
\end{enumerate}
\end{enumerate}
\end{theo}

\begin{proof} (1) (i)$\Rightarrow$(ii) If (i) holds, then $\A$ is a (semi)hereditary category. Hence every (finitely generated) projective object of $\A$ is self-Rickart \cite[Theorem~4.7]{CK}, and thus self-CS-Rickart. Also, every (finitely generated) projective object of $\A$ is weak duo by hypothesis. Now the conclusion follows by Corollary \ref{st00}.

(ii)$\Rightarrow$(i) If (ii) holds, then every (finitely generated) projective object of $\A$ has SIP by \cite[Corollary~2.2]{OHS}, whence we deduce that $\A$ is a (semi)hereditary category by \cite[Proposition~1.8]{Garcia}, both cited results having proofs which are valid in our setting. Then every (finitely generated) subobject of a projective object of $\A$ is projective. But every (finitely generated) projective object of $\A$ is also weak duo by hypothesis.

(2) (i)$\Rightarrow$(ii) If (i) holds, then $\A$ is a (semi)cohereditary category. Hence every (finitely cogenerated) injective object of $\A$ is dual self-Rickart \cite[Theorem~4.4]{CK}, and thus dual self-CS-Rickart. Also, every (finitely cogenerated) injective object of $\A$ is weak duo by hypothesis. Now the conclusion follows by Corollary \ref{st00}.

(ii)$\Rightarrow$(i) If (ii) holds, then every (finitely cogenerated) injective object of $\A$ has SSP by \cite[Corollary~2.2]{OHS}, whence we deduce that $\A$ is a (semi)cohereditary category by \cite[Proposition~1.7]{Garcia}, because the proofs of the cited results work in our setting. Then every (finitely cogenerated) factor object of an injective object of $\A$ is injective. But every (finitely cogenerated) injective object of $\A$ is also weak duo by hypothesis.
\end{proof}

\begin{rem} \rm The last two theorems are applicable, for instance, to module categories and comodule categories, which are known to be locally finitely generated Grothendieck.
\end{rem}

\end{document}